\documentclass[12pt,british]{article}

\usepackage{babel}
\usepackage{mathrsfs} 
\usepackage{amsmath,amsfonts,amssymb,amsthm,cases}
\usepackage{ucs} 
\usepackage[utf8x]{inputenc} 
\usepackage[T1]{fontenc} 

\usepackage{comment} 

\usepackage{enumerate} 
\usepackage{textcomp,url}
\usepackage{eucal}

\usepackage[noBBpl]{mathpazo}
\usepackage[matrix,arrow]{xy}

\usepackage[%
bookmarks=true,
breaklinks=true,
bookmarksnumbered = true,
colorlinks= true,
urlcolor= green,
anchorcolor = yellow,
citecolor=blue,
pdfborder= {0 0 0}
]{hyperref}                   

\makeatletter

\renewcommand{\p@enumii}{}

\def\@enum@{\list{\csname label\@enumctr\endcsname}%
          {\usecounter{\@enumctr}\def\makelabel##1{
\normalfont\ignorespaces\emph{{##1}~}}
\setlength{\labelsep}{3pt}
\setlength{\parsep}{0pt}
\setlength{\itemsep}{0pt}
\setlength{\leftmargin}{0pt}
\setlength{\labelwidth}{0pt}
\setlength{\listparindent}{\parindent}
\setlength{\itemsep}{0pt}
\setlength{\itemindent}{0pt}
\topsep=3pt plus 1pt minus 1 pt}}

\makeatother

\renewcommand{\epsilon}{\ensuremath{\varepsilon}}
\renewcommand{\phi}{\ensuremath{\varphi}}

\renewcommand{\to}{\ensuremath{\longrightarrow}}
\renewcommand{\mapsto}{\ensuremath{\longmapsto}}
\allowdisplaybreaks

\newcommand{\R}{\ensuremath{\mathbb R}}
\newcommand{\N}{\ensuremath{\mathbb N}}
\newcommand{\Z}{\ensuremath{\mathbb Z}}

\newcommand{\sn}[1][n]{\ensuremath{S_{{#1}}}}

\DeclareRobustCommand*{\up}[1]{\textsuperscript{#1}}

\renewcommand{\th}{\up{th}}

\newcommand{\im}[1]{\ensuremath{\operatorname{\text{Im}}\left({#1}\right)}}

\makeatletter

\def\@map#1#2[#3]{\mbox{$#1 \colon\thinspace #2 \to #3$}}
\def\map#1#2{\@ifnextchar [{\@map{#1}{#2}}{\@map{#1}{#2}[#2]}}

\makeatother

\newcommand{\brak}[1]{\ensuremath{\left\{ #1 \right\}}}
\newcommand{\ang}[1]{\ensuremath{\left\langle #1\right\rangle}}
\newcommand{\set}[2]{\ensuremath{\left\{#1 \,\mid\, #2\right\}}}
\newcommand{\setang}[2]{\ensuremath{\ang{#1 \,\mid\, #2}}}

\newcommand{\ord}[1]{\ensuremath{\left\lvert #1\right\rvert}}

\newtheoremstyle{theoremm}{}{}{\itshape}{}{\scshape}{.}{ }{}

\theoremstyle{theoremm}
\newtheorem{thm}{Theorem}
\newtheorem{lem}[thm]{Lemma}
\newtheorem{prop}[thm]{Proposition}
\newtheorem{cor}[thm]{Corollary}

\newtheoremstyle{reptheorem}{}{}{\itshape}{}{\scshape}{}{ }{\thmname{#1}\mathrm{#3}}
\theoremstyle{reptheorem}

\newtheoremstyle{remark}{}{}{}{}{\scshape}{.}{ }{}
\theoremstyle{remark}

\newtheorem*{defn}{Definition}
\newtheorem{rem}[thm]{Remark}

\newtheoremstyle{comment}{}{}{\bfseries}{}{\bfseries}{:}{ }{}
\theoremstyle{comment}

\newcommand{\reth}[1]{Theorem~\protect\ref{th:#1}}
\newcommand{\relem}[1]{Lemma~\protect\ref{lem:#1}}
\newcommand{\repr}[1]{Proposition~\protect\ref{prop:#1}}
\newcommand{\reco}[1]{Corollary~\protect\ref{cor:#1}}
\newcommand{\resec}[1]{Section~\protect\ref{sec:#1}}
\newcommand{\rerem}[1]{Remark~\protect\ref{rem:#1}}

\newcommand{\req}[1]{equation~(\protect\ref{eq:#1})}

\begin{document}

\title{Bieberbach groups and flat manifolds with finite abelian holonomy from Artin braid groups} 

\author{OSCAR~OCAMPO~\\
Departamento de Matem\'atica - Instituto de Matem\'atica e Estat\'istica,\\
Universidade Federal da Bahia,\\
CEP:~40170-110 - Salvador - Ba - Brazil.\\
e-mail:~\url{oscaro@ufba.br}
}

\maketitle

\begin{abstract}
\noindent
\emph{
Let $n\geq 3$. In this paper we show that for any finite abelian subgroup $G$ of $\sn$ the crystallographic group $B_n/[P_n,P_n]$ has Bieberbach subgroups $\Gamma_{G}$ with holonomy group $G$.  
Using this approach we obtain an explicit description of the holonomy representation of the Bieberbach group $\Gamma_{G}$. As an application,  when the holonomy group is cyclic of odd order, we study the holonomy representation of $\Gamma_{G}$ and determine the existence of Anosov diffeomorphisms and K\"ahler geometry of the flat manifold ${\cal X}_{\Gamma_{G}}$ with fundamental group the Bieberbach group $\Gamma_{G}\leq B_n/[P_n,P_n]$.
}
 \end{abstract}
 
\let\thefootnote\relax

\section{Introduction}\label{sec:intro}

A surprising connection  was given in \cite{GGO} between crystallographic groups with quotients of the Artin braid group $B_n$ by the commutator subgroup of the pure Artin braid group $P_n$, $B_n/[P_n,P_n]$. 
That results were extended to the generalised braid group associated to an arbitrary complex reflection group and also was studied the existence of K\"ahler structures on flat manifolds arising from subquotients of these groups, see \cite{Ma}.
Was proved in \cite[Proposition~1]{GGO} that, for $n\geq 3$, $B_n/[P_n,P_n]$ is a crystallographic group of dimension $n(n-1)/2$, with holonomy group the symmetric group $S_n$, such that the torsion of the quotient $B_n/[P_n,P_n]$ is equal to the odd torsion of $\sn$ (see \cite[Corollary~4]{GGO}). As a consequence $\widetilde{H}=\frac{\sigma^{-1}(H)}{[P_n,\, P_n]}$ is a Bieberbach subgroup of $B_n/[P_n, P_n]$ with holonomy group $H$ if and only if $H$ is a 2-subgroup of $S_n$, where $\sigma\colon B_n \rightarrow S_n$ is  the natural projection, see \cite[Corollary~13~and~Theorem~20]{GGO}. 
In \cite{OR} was studied Bieberbach groups of the form $\widetilde{H}$ arising from Artin braid groups with cyclic holonomy group $H=\Z_{2^d}$, the authors computed the center of the Bieberbach group $\widetilde{H}$, decomposed its holonomy representation in irreps, and with this information the authors were able to determine whether the related flat manifold admits Anosov diffeomorphism and/or K\"ahler geometry (the last one for even dimension).

In this paper, we show that the quotient group $B_n/[P_n,P_n]$ has Bieberbach subgroups of dimension $n(n-1)/2$ with holonomy a finite abelian subgroup of $\sn$. 
Our initial motivation emanates from the fact that is not possible to construct Bieberbach subgroups of $B_n/[P_n,P_n]$ just taking $\widetilde{H}=\frac{\sigma^{-1}(H)}{[P_n,\, P_n]}$ if $H$  has odd order elements, 
 because the torsion of the quotient $B_n/[P_n,P_n]$ is equal to the odd torsion of $\sn$ (see \cite[Corollary~4]{GGO}). However, in \cite[Remarks~26(c)]{GGO} the authors showed  a Bieberbach subgroup $L$ of $B_3/[P_3,P_3]$ with holonomy group $\Z_3$. We shall follow the idea used in \cite[Remarks~26(c)]{GGO} to exhibit Bieberbach subgroups of $B_n/[P_n,P_n]$ with finite abelian holonomy $G\leq \sn$.

This paper is organised as follows. In \resec{prelim} we recall some definitions and facts about crystallographic and Bieberbach groups, Artin braid groups and the quotients related with crystallographic groups. 
In \resec{braidversion} we will prove the following theorem that guarantees the existence and realise Bieberbach subgroups of $B_n/[P_n,P_n]$ with abelian holonomy group contained in $\sn$. 

\begin{thm}\label{th:akthm}
Let $G$ be a finite abelian group. 
\begin{enumerate}
 \item\label{it:akthma} There exists $n$ and a Bieberbach subgroup $\Gamma_G$ of $B_n/[P_n,P_n]$ of dimension $n(n-1)/2$ 
with holonomy group $G$.

 \item\label{it:akthmb} The finite abelian group $G$ is the holonomy group of a flat manifold $\mathcal{X}_{\Gamma_{G}}$ of dimension $n(n-~1)/2$, 
where $n$ is an integer for which $G$ embeds in the symmetric group $\sn$, and 
the fundamental group of $\mathcal{X}_{\Gamma_{G}}$ is isomorphic to a subgroup  $\Gamma_G$ of $B_n/[P_n, P_n]$.

\end{enumerate}
\end{thm}

L.~Auslander and M.~Kuranishi proved that if $G$ is any finite group and $G=F/R$, with $F$ and $R$ free non-abelian, then $F_0=F/[R,R]$ is the fundamental group of a flat manifold (see \cite[Theorem~3]{AK}). As a consequence, they proved that any finite group $G$ is the holonomy group of a flat manifold, see \reth{TeoAK}. 
We note that \reth{akthm} reproves Auslander-Kuranishi theorem for finite abelian groups using braid theory. 

As an application, motivated by \cite{OR}, in \resec{zq} we shall consider the case of cyclic holonomy groups of odd order, i.e. Bieberbach subgroups $\Gamma_{\Z_q}$ ($\Gamma_{q}$ for short) with holonomy group $\Z_q\leq \sn$ of odd order.  
In Subsection~\ref{subsec:matrix} we describe the holonomy representation of $\Gamma_{q}$ as a matrix representation (see \reth{mainrep}) and using it we prove the following.

\begin{thm}\label{th:mainzpr}
Let $q=p_1^{r_1} p_2^{r_2}\cdots p_t^{r_t}$ be an odd number, where $p_i^{r_i}$ are distinct odd primes and $r_i\geq 1$ for all $1\leq i\leq t$. 
Let $\mathcal{X}_{\Gamma_{q}}$ be the  flat manifold of dimension $\frac{n(n-1)}{2}$ with fundamental group 
$\Gamma_{q}\leq B_{n}/[P_{n}, P_{n}]$ and holonomy group $\Z_{q}$. Then
\begin{enumerate}
 \item $\mathcal{X}_{\Gamma_{q}}$ is orientable.
 \item 
The first Betti number of $\mathcal{X}_{\Gamma_{q}}$ is $\beta_{1}(\mathcal{X}_{\Gamma_{q}})=\displaystyle\sum_{i=1}^t\frac{p_i^{r_i}-1}{2} + \frac{t(t-1)}{2}$.

\item The flat manifold $\mathcal{X}_{\Gamma_{q}}$ with fundamental group $\Gamma_{q}$ admits Anosov diffeomorphism if and only if $q\neq 3$.
\end{enumerate}
\end{thm}

A long standing problem in dynamical systems is the classification of all compact manifolds supporting an Anosov diffeomorphism \cite{Sm}. 
We note that the proof of \reth{mainzpr} mostly depends on the holonomy representation of the Bieberbach group $\Gamma_{q}$, exploring the eigenvalues of the matrix representation (see \req{matrixzpr} in \reth{mainrep}).
At the end of Subsection~\ref{subsec:matrix} we give a presentation of the Bieberbach group $\Gamma_q$ of \reth{mainzpr} (see \repr{pres}) and we also give a set of generators of its center in \reth{centerzdk}. 
Finally, in Subsection~\ref{subsec:zpr} we obtain the following result.

\begin{thm}\label{th:kahler}
Let $p$ be an odd prime and let $r\geq 1$. 
Let $\mathcal{X}_{\Gamma_{p^r}}$ be the  flat manifold of dimension $\frac{p^r(p^r-1)}{2}$ with fundamental group 
$\Gamma_{p^r}\leq B_{p^r}/[P_{p^r}, P_{p^r}]$ and holonomy group $\Z_{p^r}$. Then the flat manifold  $\mathcal{X}_{\Gamma_{p^r}}$ is K\"ahler if and only if there is an integer $u\geq 1$ such that $p^r=4u+1$.
\end{thm}

As a consequence if $p$ is an odd prime such that $p^r=4u+1$ for some $u,r\geq1$ then the flat manifold  $\mathcal{X}_{\Gamma_{p^r}}$ of \reth{kahler} is a Calabi-Yau flat manifold of dimension $2u(4u+1)$, see \reco{kahler}.



\section{Preliminaries}\label{sec:prelim}

In this section, we recall briefly the definition of crystallographic and Bieberbach groups, for more details, see~\cite[Section~I.1.1]{Charlap},~\cite[Section~2.1]{Dekimpe} or~\cite[Chapter~3]{Wolf}. We also recall some facts about the Artin braid group $B_{n}$ on $n$ strings as well as the quotient group $B_n/[P_n,P_n]$ studied in \cite{GGO}. We refer the reader to~\cite{Ha} or \cite{KM} for more details about Artin braid groups. From now on, we identify $\operatorname{Aut}(\Z^n)$ with $\operatorname{GL}(n,\Z)$.

\subsection{Crystallographic and Bieberbach groups}\label{subsec:cryst}

 Let $G$ be a Hausdorff topological group. A subgroup $H$ of $G$ is said to be a \emph{discrete subgroup} if it is a discrete subset. If $H$ is a closed subgroup of $G$ then the coset space $G/H$ has the quotient topology for the projection $\map{\pi}{G}[G/H]$, and $H$ is said to be \emph{uniform} if $G/H$ is compact.

\begin{defn}\label{DefCristGeo}
A discrete and uniform subgroup  $\Pi$ of $\R^n\rtimes \operatorname{O}(n,\R)\subseteq \operatorname{Aff}(\R^n)$ is said to be a \emph{crystallographic group of dimension $n$}. If in addition $\Pi$ is torsion free then $\Pi$ is called a \emph{Bieberbach group} of dimension $n$.
\end{defn}

An  \emph{integral representation of rank $n$ of  $\Phi$} is a homomorphism $\map{\Theta}{\Phi}[\operatorname{Aut}(\Z^n)]$.  We say that $\Theta$ is a \emph{faithful representation} if it is injective.
The following characterisation of crystallographic groups was used in \cite{GGO} to stablish the connection between braid groups and crystallographic groups. 

\begin{lem}[Lemma~8 of \cite{GGO}]\label{lem:DefC2}
Let $\Pi$ be a group. Then $\Pi$ is a crystallographic group if and only if there exist an integer $n\in \N$ and a short exact sequence 
$\xymatrix{
0 \ar[r] & \Z^n \ar[r] & \Pi \ar[r]^{\zeta} & \Phi \ar[r] & 1}	
$
such that:
\begin{enumerate}
\item\label{it:DefC2a} $\Phi$ is finite, and
\item\label{it:DefC2b} the integral representation $\map{\Theta}{\Phi}[\operatorname{Aut}(\Z^n)]$, induced by 
conjugation on $\Z^n$ and defined by $\Theta(\phi)(x)=\pi x \pi^{-1}$, where $x\in \Z^{n}$, $\phi\in \Phi$ and $\pi\in \Pi$ is such that $\zeta(\pi)=\phi$, is faithful. 
\end{enumerate}
\end{lem}

If $\Pi$ is a crystallographic group, the integer $n$ that appears in the statement of~\relem{DefC2} is called the \emph{dimension} of $\Pi$, the finite group $\Phi$ is called the \emph{holonomy group} of $\Pi$, and the integral representation $\map{\Theta}{\Phi}[\operatorname{Aut}(\Z^{n})]$ is called the \emph{holonomy representation of  $\Pi$}.

Was stated in \cite[Corollary~10]{GGO} that if $\Pi$ is a crystallographic group of dimension $n$ and holonomy group $\Phi$, and $H$ is a subgroup of $\Phi$, then there exists a crystallographic subgroup of $\Pi$ of dimension $n$ with holonomy group $H$.
This statement was used in \cite[Corollary~13]{GGO} to provide other crystallographic subgroups of $B_n/[P_n,P_n]$.

A Riemannian manifold $M$ is called \emph{flat} if it has zero curvature at every point.  
As a consequence of the first Bieberbach Theorem, there is a correspondence between Bieberbach groups and fundamental groups of closed flat Riemannian manifolds (see~\cite[Theorem~2.1.1]{Dekimpe} and the paragraph that follows it). 

\begin{rem}\label{rem:orientable}
We recall that the flat manifold determined by a Bieberbach group $\Pi$ is orientable if and only if the integral representation $\map{\Theta}{\Phi}[\operatorname{GL}(n,\Z)]$ satisfies $\im{\Theta}\subseteq \operatorname{SO}(n,\Z)$. This being the case, we say that $\Pi$ is \emph{an orientable Bieberbach group}. 
\end{rem}

It is a natural problem to classify the finite groups that are the holonomy group of a flat manifold. The answer was given by L.~Auslander and M.~Kuranishi in 1957, see \cite[Theorem~3]{AK}, \cite[Theorem~3.4.8]{Wolf} or \cite[Theorem~III.5.2]{Charlap}.

\begin{thm}[Auslander and Kuranishi]\label{th:TeoAK}
Any finite group is the holonomy group of some flat manifold.
\end{thm}

\subsection{Artin braid groups}\label{subsec:braids}

We recall in this subsection some facts about the Artin braid group $B_{n}$ on $n$ strings. 
We refer the reader to~\cite{Ha} and \cite{KM} for more details. It is well known 
that $B_{n}$ possesses a presentation with generators $\sigma_{1},\ldots,\sigma_{n-1}$ that are subject to the relations $\sigma_{i} \sigma_{j} = \sigma_{j}  \sigma_{i}$ for all $1\leq i<j\leq n-1$ such that $\ord{i-j}\geq 2$
and 
$\sigma_{i+1} \sigma_{i} \sigma_{i+1} =\sigma_{i} \sigma_{i+1} \sigma_{i}$ for all $1\leq i\leq n-2$.

Let $\map{\sigma}{B_{n}}[\sn]$ be the homomorphism defined on the given generators of $B_{n}$ by $\sigma(\sigma_{i})=(i,\, i+1)$ for all $1\leq i\leq n-1$. Just as for braids, we read permutations from left to right so that if $\alpha, \beta \in \sn$ then their product is defined by $\alpha\cdot \beta (i)=\beta(\alpha(i))$ for $i=1,2,\ldots, n$. The pure braid group $P_{n}$ on $n$ strings is defined to be the kernel of $\sigma$, from which we obtain the following short exact sequence:
\begin{equation}\label{eq:sespn}
1 \to P_n \to  B_n \stackrel{\sigma}{\to} \sn \to 1.
\end{equation}
A generating set of $P_{n}$ is given by $\brak{A_{i,j}}_{1\leq i<j\leq n}$, where $A_{i,j}=\sigma_{j-1}\cdots \sigma_{i+1}\sigma_{i}^{2} \sigma_{i+1}^{-1}\cdots \sigma_{j-1}^{-1}$.
For $1\leq i<j\leq n$, we also set $A_{j,i}=A_{i,j}$, 
and if $A_{i,j}$ appears with exponent $m_{i,j}\in \Z$, then we let $m_{j,i}=m_{i,j}$. 
It follows from the presentation of $P_{n}$ given  in~\cite{Ha} that $P_n/[P_n,P_n]$ is isomorphic to $\Z^{n(n-1)/2}$, and that a basis is given by the $A_{i,j}$, where $1\leq i<j\leq n$, and where by abuse of notation, the $[P_n, P_n]$-coset of $A_{i,j}$ will also be denoted by $A_{i,j}$. 
Using \req{sespn}, we obtain the following short exact sequence:
\begin{equation}\label{eq:sespnquot}
1 \to  P_n/[P_n,P_n] \to  B_n/[P_n,P_n] \stackrel{\overline{\sigma}}{\to} \sn \to 1,
\end{equation}
where $\map{\overline{\sigma}}{B_n/[P_n,P_n]}[\sn]$ is the homomorphism induced by $\sigma$. 

In \cite[Theorem~3]{GGO} it is assumed that  $n_1,\ldots,n_t$  are odd  
integers.  Nevertheless the proof of \cite[Theorem~3(a)]{GGO}, which states  that  the  inclusion of a direct product 
\break 
$\map{\iota}{B_{n_{1}}\times \cdots \times B_{n_{t}}}[B_{m}]$ induces an injective homomorphism 
\break
$\map{\overline{\iota}}{\displaystyle\frac{B_{n_1}}{[P_{n_1},P_{n_1}]}\times \cdots \times
  \frac{B_{n_t}}{[P_{n_t},P_{n_t}]}}[\frac{B_{m}}{[P_{m},P_{m}]}]$, do not use the hypotheses that $n_1,n_2,\ldots,n_t$ are odd integers. So we can state that result as follows.

\begin{thm}[Theorem~3(a) of \cite{GGO}]\label{th:inclusion}
Let $t,m\in \N$, and let $n_1,n_2,\ldots,n_t$ be integers greater than or equal to $3$ for which $\sum_{i=1}^{t} \, n_i\leq m$. Then the inclusion $\map{\iota}{B_{n_{1}}\times \cdots \times B_{n_{t}}}[B_{m}]$ induces an injective homomorphism
$\map{\overline{\iota}}{\displaystyle\frac{B_{n_1}}{[P_{n_1},P_{n_1}]}\times \cdots \times \frac{B_{n_t}}{[P_{n_t},P_{n_t}]}}[\frac{B_{m}}{[P_{m},P_{m}]}]$.
\end{thm}
\reth{inclusion} will be useful to prove \repr{inclusion}.
Since $B_1$ is the trivial group and $B_2/[P_2, P_2]\cong \Z$, we shall suppose in most of this paper that $n\geq 3$. 
The study of the action by conjugacy of $B_n$ on $P_n$ provides the following result. 

\begin{prop}[Proposition~12 of \cite{GGO}]\label{prop:1}
  Let $\alpha\in B_n/[P_n, P_n]$, and let $\pi$ be the permutation induced by $\alpha^{-1}$, then $\alpha A_{i,j}\alpha^{-1}=A_{\pi(i), \pi(j)}$
  in $P_n/[P_n, P_n]$.
\end{prop}

\repr{1} was important to establish the connection between braid groups and crystallographic groups, because from it we know explicitly the holonomy representation of the quotient group $B_n/[P_n,P_n]$ and then applying \relem{DefC2} was proved that there is the short exact sequence of \req{sespnquot},
and the middle group $B_n/[P_n,P_n]$ is a crystallographic group (see \cite[Proposition~1]{GGO}).

Using~\cite[Proposition~1 and Corollary~10]{GGO} it is possible to produce other crystallographic groups as follows. 
Let  $H$ be a subgroup of $\sn$, and consider the following short exact sequence  
$1 \to \frac{P_n}{[P_n, P_n]} \to \widetilde{H}_n \stackrel{\overline{\sigma}}{\to} H \to 1
$
induced by the sequence given in \req{sespnquot}, where $\widetilde{H}_n$ is defined by 
\begin{equation}\label{eq:BCG3}
\widetilde{H}_n=\frac{\sigma^{-1}(H)}{[P_n, P_n]}.
\end{equation}
Then the group $\widetilde{H}_n$ defined by \req{BCG3}
is a crystallographic group of dimension $n(n-1)/2$ with holonomy group $H$, see \cite[Corollary~13]{GGO}. 
Was proved in \cite[Theorem~2]{GGO} that if $n\geq 3$ then the quotient group $B_n/[P_n, P_n]$ has no finite-order elements of even order. Hence using the information above was proved that if $H$ is a finite $2$-group, then $H$ is the holonomy group of some flat manifold $M$. Further, the dimension of $M$ may be chosen to be $n(n-1)/2$, where $n$ is an integer for which $H$ embeds in the symmetric group $\sn$, and the fundamental group of $M$ is isomorphic to the subgroup $\widetilde{H}_n$ of $B_n/[P_n, P_n]$,
see \cite[Theorem~20]{GGO}.

\section{Bieberbach groups with abelian holonomy group from Artin braid groups}\label{sec:braidversion}

Let $G$ be a finite abelian group. In this section we shall prove \reth{akthm} (it also reproves \reth{TeoAK} for abelian groups using braid theory) that guaranteed that  there exist $n$ and a Bieberbach subgroup $\Gamma_G$ of $B_n/[P_n,P_n]$ of dimension $n(n-1)/2$ 
with holonomy group $G$. As a consequence $G$ is the holonomy group of a flat manifold $\mathcal{X}_{\Gamma_{G}}$ of dimension $n(n-1)/2$, 
where $n$ is an integer for which $G$ embeds in the symmetric group $\sn$, and 
the fundamental group of $\mathcal{X}_{\Gamma_{G}}$ is isomorphic to a subgroup  $\Gamma_G$ of $B_n/[P_n, P_n]$.

\begin{proof}[Proof of \reth{akthm}]
Let $q$ be the cardinality of the finite abelian group $G$, $\left| G \right|=q$.
\begin{enumerate}
 \item We will consider 3 cases for $q$ depending of its prime decomposition. 

 \begin{enumerate}
  \item Suppose that $q$ is given by product of powers of the prime 2, so $G$ is a 2-group.
  In this case the result follows from \cite[Theorem~20]{GGO}.

  \item Suppose that $q$ is even and that in the prime decomposition of $q$ appears at least one odd prime. 
Then, from the fundamental theorem of finite abelian groups, the group $G$ may be decomposed as $G=G_1\times G_2$, with 
\begin{align}
G_1&=\Z_{p_1^{r_1}}\times \Z_{p_2^{r_2}}\times \cdots \times \Z_{p_t^{r_t}}, \textrm{ and }\label{eq:g1}\\ 
G_2&=\Z_{2^{r_{t+1}}}\times \cdots \times \Z_{2^{r_{t+v}}}\nonumber
\end{align}
where $p_i$ are odd primes not necessarily different. Let $d=\operatorname{lcm}(p_1^{r_1}, p_2^{r_2},\ldots,p_t^{r_t})$ the least common multiple of the integers $p_1^{r_1}, p_2^{r_2},\ldots,p_t^{r_t}$.

For $1\leq l\leq t$, let $n_l=\sum_{i=1}^lp_i^{r_i}$, and for $1\leq f\leq v$ let 
$m_f=\sum_{i=1}^{f}2^{r_{t+i}}$.
So from \cite[Remark]{Ho} $G$ is a subgroup of $\sn$, with $n_t+m_v=n$. 
For convention, let $n_0=0$.

\textbf{STEP 1}. 
Consider the following subsets of $B_n/[P_n,P_n]$, 
\begin{align}
X_1&=\left\{ A_{1,2}\delta_{0,n_1},\, A_{n_1+1,n_1+2}\delta_{n_1,p_2^{r_2}},\, \cdots,\, A_{n_{t-1}+1,n_{t-1}+2}\delta_{n_{t-1},p_t^{r_t}}\right\}\label{eq:x1}\\
X_2&=\left\{ A_{r,s}^{d} \textrm{ for } 1\leq r<s\leq n_t \right\}\label{eq:x2}\\
X_3&=\left\{ A_{i,j} \textrm{ for } 1\leq i<j\leq n \textrm{ and } (i,j)\notin \set{(r,s)}{1\leq r<s\leq n_t} \right\}\label{eq:x3}\\
X_4&=\left\{\eta_{1},\, \eta_{2},\, \ldots,\,  \eta_{v} \right\}\nonumber
\end{align}
where 
\begin{equation}\label{eq:delta}
\delta_{n_l,\, p_{l+1}^{r_{l+1}}}=\sigma_{n_l+p_{l+1}^{r_{l+1}}-1}\cdots \sigma_{n_l+\frac{p_{l+1}^{r_{l+1}}+1}{2}} \sigma_{n_l+\frac{p_{l+1}^{r_{l+1}}-1}{2}}^{-1}\cdots \sigma_{n_l+1}^{-1},
\end{equation} 
for $0\leq l\leq t-1$, is the element given in \cite[Equation~(14)]{GGO}
and
\begin{equation*}\label{eq:eta}
\eta_f=\overline{\iota}(1,\ldots,1,\sigma_{1}\cdots \sigma_{2^{r_{t+f}}-1},1,\ldots,1)=\sigma_{n_t+m_{f-1}+1}\cdots \sigma_{n_t+m_f-1}
\end{equation*} 
for $1\leq f\leq v$, where 
\begin{equation}\label{eq:iotabar}
\map{\overline{\iota}}{\displaystyle\frac{B_{p_1^{r_1}}}{[P_{p_1^{r_1}},P_{p_1^{r_1}}]}\times \cdots \times \frac{B_{p_t^{r_t}}}{[P_{p_t^{r_t}},P_{p_t^{r_t}}]}\times \frac{B_{2^{r_{t+1}}}}{[P_{2^{r_{t+1}}},P_{2^{r_{t+1}}}]}\times \cdots \times \frac{B_{2^{r_{t+v}}}}{[P_{r_{t+v}},P_{r_{t+v}}]}}[\frac{B_{n}}{[P_{n},P_{n}]}]
\end{equation} 
is the injective homomorphism from \reth{inclusion} and the element $\sigma_{1}\cdots \sigma_{2^{r_{t+f}}-1}$ is in the $(t+f)$\th factor. 
Recall from \cite[Lemma~28]{GGO} that $\delta_{n_l,\, p_{l+1}^{r_{l+1}}}$ has order $p_{l+1}^{r_{l+1}}$ in $B_n/[P_n,P_n]$.

Let $\Gamma$ be the subgroup of $\frac{\sigma^{-1}\left( G \right)}{[P_n,\ P_n]}$ 
generated by $X=X_1\cup X_2\cup X_3\cup X_4$. 
Let $L=\Gamma\cap \operatorname{Ker}(\overline{\sigma})=\Gamma\cap \setang{A_{i,j}}{1\leq i<j\leq n}$. So, $L$ is a free abelian group. 
We shall construct a basis for $L$.

Let 
\begin{align}
\Lambda_{\textrm{[odd prime]}}(s_1,\ldots,s_t)&= (A_{1,2}\delta_{0,n_1})^{s_1}(A_{n_1+1,n_1+2}\delta_{n_1,p_2^{r_2}})^{s_2}\cdots(A_{n_{t-1}+1,n_{t-1}+2}\delta_{n_{t-1},p_t^{r_t}})^{s_t} \label{eq:lambdaodd}\\
\Lambda_{\textrm{[prime 2]}}(s_{t+1},\ldots,s_{t+v})&=\eta_{1}^{s_{t+1}}\eta_{2}^{s_{t+2}} \ldots \eta_{v}^{s_{t+v}}\nonumber.
\end{align}

We need a set of coset representatives of $L$ in $\Gamma$:
\begin{align*}
M & =\left\{ \prod_{\substack{ 0\leq s_\alpha\leq p_\alpha^{r_\alpha}-1 \\ 1\leq \alpha\leq t }} \Lambda_{\textrm{[odd prime]}}(s_1,\ldots,s_t) 
\prod_{\substack{ 0\leq s_\alpha\leq 2^{r_\alpha}-1 \\ t+1\leq \alpha\leq t+v }}\Lambda_{\textrm{[prime 2]}}(s_{t+1},\ldots,s_{t+v})
\right\}       
\end{align*}
From the Reidemeister-Schreier method (see \cite[Appendix~I]{KM}) the elements
\begin{equation*}
 y_{j,k}=M_kx_j\overline{M_kx_j}^{-1}
\end{equation*}
form a set of generators for $L$, with $x_j\in X$ and $M_k\in M$. Now, we analyse cases for different choices of $x_j\in X$.

\begin{enumerate}
	\item[Case 1.]  We note that for $x_j\in X_2\cup X_3$, we obtain
 $\overline{M_kx_j}=\overline{M_k}=M_k$.
 Then, in these cases we obtain generators $X_2\cup X_3$. 

\item[Case 2.] For $x_j=A_{1,2}\delta_{0,n_1}$ and for 
\begin{equation*}
M_k= \prod_{\substack{ 0\leq s_\alpha\leq p_\alpha^{r_\alpha}-1 \\ 2\leq \alpha\leq t }} \Lambda_{\textrm{odd prime}}(n_1-1 ,s_2,\ldots,s_t) 
\prod_{\substack{ 0\leq s_\alpha\leq 2^{r_\alpha}-1 \\ t+1\leq \alpha\leq t+v }}\Lambda_{\textrm{prime 2}}(s_{t+1},\ldots,s_{t+v})
\end{equation*} 
we get 
$y_{j,k}=M_kx_j\overline{M_kx_j}^{-1}=(A_{1,2}\delta_{0,n_1})^{p_1^{r_1}}$. 
For other cases of $M_k$, still taking $x_j=A_{1,2}\delta_{0,n_1}$, we get $y_{j,k}=e$, since $\overline{M_kx_j}=M_kx_j$. 

A similar argument for $x_j=A_{n_l+1,n_l+2}\delta_{n_l,p_{l+1}^{r_{l+1}}}$, $1\leq l\leq t-1$, is used to show that we obtain generators
$(A_{n_l+1,n_l+2}\delta_{n_l,p_{l+1}^{r_{l+1}}})^{p_{l+1}^{r_{l+1}}}$ for  $1\leq l\leq t-1$

\item[Case 3.] Now we consider the case $x_j=\eta_f$, with $f\in\brak{1,2,\ldots,v}$. Let $M_k$ be the expression:
\begin{equation*}
M_k=\prod_{\substack{ 0\leq s_\alpha\leq p_\alpha^{r_\alpha}-1 \\ 1\leq \alpha\leq t }} \Lambda_{\textrm{odd prime}}(s_1 ,s_2,\ldots,s_t) 
\prod_{\substack{ 0\leq s_\alpha\leq 2^{r_\alpha}-1 \\ t+1\leq \alpha\leq t+v \\\alpha\neq t+f}}\Lambda_{\textrm{prime 2}}(s_{t+1},\ldots,2^{r_{t+f}}-1,\ldots,s_{t+v})
\end{equation*}
So, in these cases we get $y_{j,k}=\eta_f^{2^{r_{t+f}}}$. 
Since $(\sigma_{1}\cdots \sigma_{2^{r_{t+f}}-1})^{2^{r_{t+f}}}=\Delta^2_{2^{r_{t+f}}}$ is the full twist in $\frac{B_{2^{r_{t+f}}}}{[P_{2^{r_{t+f}}},P_{2^{r_{t+f}}}]}$ then from the inclusion given in \req{iotabar} we obtain $\eta_f^{2^{r_{t+f}}}=\displaystyle\prod_{n_t+m_{f-1}+1\leq i<j\leq n_t+m_f}A_{i,j}$ in $B_n/[P_n,P_n]$ is a product of elements of $X_3$. 
 For other cases of $M_k$ we get $y_{j,k}=e$. 
\end{enumerate}

Hence $L$ can be generated by $C_1\cup X_2\cup X_3$, where $X_2$ and $X_3$ was given in \req{x2} and \req{x3}, respectively, and
\begin{align}\label{eq:a1}
C_1&=\left\{ (A_{1,2}\delta_{0,n_1})^{p_1^{r_1}},\, (A_{n_1+1,n_1+2}\delta_{n_1,p_2^{r_2}})^{p_2^{r_2}},\,\cdots,\, (A_{n_{t-1}+1,n_{t-1}+2}\delta_{n_{t-1},p_t^{r_t}})^{p_t^{r_t}}\right\}
\end{align}
Let $C_3=X_3$ (see \req{x3}).
Let $I$ denote the index set
\begin{align*}
I&=\set{(i,j)}{1\leq i<j\leq n_t \textrm{ with } (i,j)\neq (n_l+1, n_l+2) \textrm{ for } l=0,1,\cdots,t-1},
\end{align*}
and let 
\begin{equation}\label{eq:oddbasis}
C_2=\left\{ A_{r,s}^{d}\in X_2 \mid (r,s)\in I \right\}.
\end{equation}
Hence, the set
\begin{equation}\label{eq:basis}
C_1\cup C_2\cup C_3 
\end{equation}
is a basis for the free abelian group $L$.

\textbf{STEP 2}. 
Now, we prove that $\Gamma$ is torsion free. 
Suppose on the contrary that $w$ is a non-trivial torsion element of $\Gamma$. 
First we observe that, since $B_n/[P_n,P_n]$ has no $2$-torsion \cite[Theorem~2]{GGO} 
then, in case $w$ is a non-trivial finite order element its torsion is odd. Furthermore, since $L$ is torsion free then $w\notin L$ and it projects onto an element of odd order in $G_1$. Recall that $d=\operatorname{lcm}(p_1^{r_1}, p_2^{r_2},\ldots,p_t^{r_t})$.
Hence there exists $\theta \in L$ and for $0\leq l\leq t-1$ there exist $x_{l+1}\in \brak{0,1,\ldots,p_{l+1}^{r_{l+1}}-1}$
with $(x_1,x_2,\cdots,x_t)\neq (0,0,\cdots,0)$ such that 
\begin{equation}\label{eq:w}
w=\theta (A_{1,2}\delta_{0,n_1})^{x_1} (A_{n_1+1,n_1+2}\delta_{n_1,p_2^{r_2}})^{x_2}\cdots (A_{n_{t-1}+1,n_{t-1}+2}\delta_{n_{t-1},p_t^{r_t}})^{x_t}. 
\end{equation}
The elements $A_{i,j}$ with $n_t+1\leq i<j\leq n$ does not belong to $\theta$ because $w$ projects onto an element of odd order in $G_1$. 

Let $A=(A_{1,2}\delta_{0,n_1})^{x_1} (A_{n_1+1,n_1+2}\delta_{n_1,p_2^{r_2}})^{x_2}\cdots (A_{n_{t-1}+1,n_{t-1}+2}\delta_{n_{t-1},p_t^{r_t}})^{x_t}$. 
Recall, from the expression of $w$ given in \req{w}, that $(x_1,x_2,\cdots,x_t)\neq (0,0,\cdots,0)$, so at least one the $x_i$'s is different from zero. We will prove the case $x_1\neq 0$, and point out that the other cases follows in a similar way.

From $\delta_{0,p_1^{r_1}}^{p_1^{r_1}}=1$ in $B_n/[P_n,P_n]$ we conclude that
\begin{align*}
 (A_{1,2}\delta_{0,p_1^{r_1}})^{p_1^{r_1}}& = A_{1,2}\, \delta_{0,p_1^{r_1}}A_{1,2}\delta_{0,p_1^{r_1}}^{-1}\, \delta_{0,p_1^{r_1}}^{2}A_{1,2}\delta_{0,p_1^{r_1}}^{-2}\cdots \delta_{0,p_1^{r_1}}^{p_1^{r_1}-1} A_{1,2}\delta_{0,p_1^{r_1}}^{-(p_1^{r_1}-1)} \, \delta_{0,p_1^{r_1}}^{p_1^{r_1}} \\
&= A_{1,2}A_{1,p_1^{r_1}}A_{p_1^{r_1}-1,p_1^{r_1}}\cdots A_{2,3}.
\end{align*}
Therefore
\begin{equation}\label{eq:aq}
A^d= A_{1,2}^{\left(\frac{d}{p_1^{r_1}}\right)x_1}A_{1,p_1^{r_1}}^{\left(\frac{d}{p_1^{r_1}}\right)x_1}A_{p_1^{r_1}-1,p_1^{r_1}}^{\left(\frac{d}{p_1^{r_1}}\right)x_1}\cdots A_{2,3}^{\left(\frac{d}{p_1^{r_1}}\right)x_1} \cdot  
\prod_{l=1}^{t-1} (A_{n_l+1,n_l+2}\delta_{n_l,p_{l+1}^{r_{l+1}}})^{d x_{l+1}}
\end{equation}

Since $\theta\in L$ then using the basis of $L$ given in \req{basis} there are $\lambda_{i,j}\in \Z$ for $1\leq i<j\leq n$ such that $\theta$ is equal to the product
\begin{eqnarray}\label{eq:theta}
\theta=(A_{1,2}\cdots A_{2,3} )^{\lambda_{1,2}} \cdot\prod_{l=1}^{t-1} \left((A_{n_{l}+1,n_{l}+2}\delta_{n_{l},p_{l+1}^{r_{l+1}}})^{p_{l+1}^{r_{l+1}}}\right)^{\lambda_{n_{l}+1,n_{l}+2}}\cdot 
\prod_{(r,s)\in I}A_{r,s}^{d\lambda_{r,s}}\cdot 
\prod_{A_{i,j}\in A_3}A_{i,j}^{\lambda_{i,j}}
\end{eqnarray}
Hence, for $w=\theta A$, we obtain
\begin{equation}\label{eq:wq1}
1=w^d=(\theta A)^d=\theta\, (A \theta A^{-1})\, (A^2\theta A^{-2})\, \cdots (A^{d-1}\theta A^{-(d-1)})\, A^d.
\end{equation}
Comparing the coefficient of $A_{1,2}$ in the expression of $w^d$, within the assumption $1~=~w^d$, see \req{aq}, \req{theta} and \req{wq1}, we obtain the following equality, in which the expression $\lambda_{1,2}$ appears $d$ times 
\begin{equation*}
\lambda_{1,2}+(\lambda_{1,2}+d\lambda_{2,3})+\cdots+(\lambda_{1,2}+d\lambda_{1,p_1^{r_1}})+\frac{d}{p_1^{r_1}}x_1=0,
\end{equation*} 
that is equivalent to the equality 
$\lambda_{1,2}+\lambda_{2,3}+\cdots+\lambda_{1,p_1^{r_1}}=-\frac{x_1}{p_1^{r_1}}$,
which has no solution in $\Z$, because $p_1^{r_1}$ does not divides $x_1$ since $x_1<p_1^{r_1}$. 
It follows that $\Gamma$ is torsion free, and so is a Bieberbach group of dimension $\frac{n(n-1)}{2}$ with holonomy group $G\subseteq \sn$.

  \item Suppose that $q$ is odd. The proof is similar to the case 
  given in item $(ii)$, so we just will appoint some minor details of the proof. 
  Since $q$ is odd we may suppose that $G=G_1$, where $G_1$ was given in \req{g1}. 
 Let $\Gamma$ be the subgroup of $\frac{B_n}{[P_n,P_n]}$ generated by $X_1$, where 
 $X_1$ is the set given in \req{x1}. 
 
 As in item $(ii)$ let $L=\Gamma\cap \operatorname{Ker}(\overline{\sigma})$. Then, taking the set of coset 
 representatives $M$ of $L$ in $\Gamma$, using only the expresion given in \req{lambdaodd} 
 and using the Reidemeister-Schreier method we may conclude that a basis for $L$ is the set $A_1\cup A_2$, see \req{a1} and \req{oddbasis}. 
The proof that $\Gamma$ is torsion free is in practice the same as the one given in 
\textbf{STEP 2} of item~$(ii)$.

\end{enumerate}

It follows that, in any case, given a finite abelian group $G$ there exists an integer $n$ and 
a Bieberbach group $\Gamma\subseteq B_n/[P_n,P_n]$ of dimension $\frac{n(n-1)}{2}$ with holonomy group $G\subseteq \sn$.

 \item Let $G$ be a finite abelian group. From item (\ref{it:akthma}) there is an integer $n$ 
 and a Bieberbach subgroup $\Gamma_G$ of $B_n/[P_n,P_n]$ of dimension $n(n-1)/2$ with 
 holonomy group $G\subseteq \sn$. 
 By applying \cite[Theorem~1]{AK} we have that there exists a manifold ${\cal X}_{\Gamma_G}$ with fundamental group $\Gamma_G$ and holonomy group $G$.
\end{enumerate}
\end{proof}

\begin{rem}
We note that, for any finite abelian group $G\leq \sn$, in \reth{akthm} we exhibit a Bieberbach group $\Gamma_G\leq B_n/[P_n,P_n]$ of dimension $n(n-1)/2$, with holonomy group $G$ and such that we may describe explicitly the holonomy representation using \repr{1}. 
It is very useful to have an explicit description of the holonomy representation of the Bieberbach group $\Gamma_G$ to deduce some geometric properties of the flat manifold ${\cal X}_{\Gamma_G}$ with fundamental group $\Gamma_G$. For instance, to see this, in  \resec{zq} we will explore in more details the Bieberbach subgroup $\Gamma_G$ of $B_n/[P_n,P_n]$ in the case in which $G$ is a cyclic group of odd order.
\end{rem}

Let $H=\Z_{n_1}\times \Z_{n_2}\times \cdots \times \Z_{n_t}$ be a finite abelian group, where $n_1,\ldots,n_t$ are powers of (not necessarily distinct) primes. Let $m$ be an integer such that $\sum_{i=1}^{t} \, n_i\leq m$. We may exhibit other Bieberbach groups with holonomy group $H$ in $\frac{B_{m}}{[P_{m},\ P_{m}]}$ of lower dimension than $m(m-1)/2$ as showed in the following result.  Let $\binom{n}{k}$ denote the binomial coefficient.

\begin{prop}\label{prop:inclusion}
Let $H=\Z_{n_1}\times \Z_{n_2}\times \cdots \times \Z_{n_t}$ be any finite abelian group, where $n_1,\ldots,n_t$ are powers of (not necessarily distinct) primes. Let $m$ be an integer such that $\sum_{i=1}^{t} \, n_i\leq m$.
Then there exist a Bieberbach subgroup of $\frac{B_{m}}{[P_{m},\ P_{m}]}$ of dimension $\sum_{i=1}^{t} \, \binom{n_i}{2}\leq \binom{m}{2}$ with holonomy group $H$. 
\end{prop}

\begin{proof}
Let $H=\Z_{n_1}\times \Z_{n_2}\times \cdots \times \Z_{n_t}$ be a finite abelian group, where $n_1,\ldots,n_t$ are powers of (not necessarily distinct) primes and let $m$ be an integer such that $\sum_{i=1}^{t} \, n_i\leq~m$.
From \reth{inclusion} we have that $\map{\overline{\iota}}{\displaystyle\frac{B_{n_1}}{[P_{n_1},P_{n_1}]}\times \cdots \times \frac{B_{n_t}}{[P_{n_t},P_{n_t}]}}[\frac{B_{m}}{[P_{m},P_{m}]}]$ is an  injective homomorphism. 
Using \reth{akthm} we construct a Bieberbach group $\Gamma_{\Z_{n_k}}$ in $\frac{B_{n_k}}{[P_{n_k},\ P_{n_k}]}$ of dimension $\binom{n_k}{2}$, for all $1\leq k\leq t$. Hence, $\overline{\iota}(\Gamma_{\Z_{n_1}}\times \cdots \times \Gamma_{\Z_{n_t}})\leq \frac{B_{m}}{[P_{m},\ P_{m}]}$ is a Bieberbach group of dimension $\sum_{i=1}^{t} \, \binom{n_i}{2}\leq \binom{m}{2}$ with holonomy group $H$.
\end{proof}

\section{$\Z_{q}$-manifolds for $q$ an odd number}\label{sec:zq}

Let $q=p_1^{r_1} p_2^{r_2}\cdots p_t^{r_t}$ be an odd number, where $\brak{p_1,\ldots,p_t}$ are distinct odd primes and $r_i\geq 1$ for all $1\leq i\leq t$. 
We shall consider the cyclic group $G$ of order $q$ isomorphic to $\Z_{p_1^{r_1}}\times \Z_{p_2^{r_2}}\times \cdots \times \Z_{p_t^{r_t}}$. 
For $1\leq l\leq t$, let $n_l=\sum_{i=1}^lp_i^{r_i}$.
From \cite[Remark]{Ho} $G$ is a subgroup of $\sn$, with $n=n_t$. 
For convention, let $n_0=0$. 
Let $\Gamma_{\Z_{q}}\leq B_{n}/[P_{n}, P_{n}]$ be the Bieberbach group constructed in \reth{akthm}. 
For short we just write $\Gamma_{q}$ to indicate the group $\Gamma_{\Z_{q}}$.
In this section we study the existence of Anosov diffeomorphisms for flat manifolds $\mathcal{X}_{\Gamma_{q}}$ with cyclic holonomy $\Z_{q}$ and fundamental group $\Gamma_{q}\leq B_{n}/[P_{n}, P_{n}]$, a quotient of the Artin braid group. We also give a presentation of the Bieberbach group $\Gamma_{q}$ and a set of generators of its center.
The Bieberbach group $\Gamma_{q}$ fits into a short exact sequence
$
1\to \Z^{\frac{n(n-1)}{2}} \to \Gamma_{q} \to \Z_{q} \to 1
$
and have holonomy representation
\begin{equation}\label{eq:repzq}
\psi_{q}\colon \Z_{q}\to \operatorname{Aut}\left(\Z^{\frac{n(n-1)}{2}}\right)
\end{equation}
induced from the action by conjugacy of $\Z_{q}$ over the group $\Z^{\frac{n(n-1)}{2}}$.

\subsection{Holonomy representation as a matrix representation and proof of \reth{mainzpr}}\label{subsec:matrix}

In the proof of \reth{akthm} were considered the subsets of $P_n/[P_n,P_n]$
\begin{align*}
A_1&=\left\{ (A_{1,2}\delta_{0,n_1})^{p_1^{r_1}},\, (A_{n_1+1,n_1+2}\delta_{n_1,p_2^{r_2}})^{p_2^{r_2}},\,\cdots,\, (A_{n_{t-1}+1,n_{t-1}+2}\delta_{n_{t-1},p_t^{r_t}})^{p_t^{r_t}}\right\}\\
A_2&=\left\{ A_{r_l,s_l}^{q} \mid (r_l,s_l)\in I \right\}
\end{align*} 
where $I$ denote the index set
\begin{align*}
I&=\set{(i,j)}{1\leq i<j\leq n \textrm{ with } (i,j)\neq (n_l+1, n_l+2) \textrm{ for } l=0,1,\cdots,t-1}.
\end{align*}
From the proof of \reth{akthm} we know that ${\cal B} = A_1\cup A_2$
is a basis of the free abelian group $L=\Z^{\frac{n(n-1)}{2}}$. 
For $0\leq l\leq t-1$ we consider the element $\delta_{n_l,\, p_{l+1}^{r_{l+1}}}$ defined in \req{delta}. 
Let 
\begin{equation*}
\delta=\prod_{0\leq l\leq t-1}\delta_{n_l,\, p_{l+1}^{r_{l+1}}}.
\end{equation*}
The element $\delta$ projects via $\overline{\sigma}\colon B_n/[P_n,P_n]\to \sn$ onto a generator of a cyclic subgroup (of order $q$) of $\sn$. 
Notice that the action of $\Z_{q}$ over the basis ${\cal B}$ may be computed using \repr{1}. Also note that 
$$
\frac{n(n-1)}{2}=\frac{(p_1^{r_1}+\cdots p_t^{r_t})(p_1^{r_1}+\cdots p_t^{r_t}-1)}{2}=\sum_{i=1}^t\frac{p_i^{r_i}(p_i^{r_i}-1)}{2} + \sum_{1\leq \ell<s\leq t}p_{\ell}^{r_{\ell}}p_s^{r_s}.
$$
 Now we order the elements of the basis ${\cal B}$ using the action of $\Z_{q}$ over ${\cal B}$:
\begin{itemize}
\item For $1\leq j\leq t$ let
$$
e_{j,0,1}=\left(A_{n_{j-1}+1,n_{j-1}+2}\delta_{n_{j-1},p_{j}^{r_{j}}}\right)^{p_{j}^{r_{j}}}
$$ 
and let
$$
e_{j,0,k+1}=\delta^{k-1} \left(A_{n_{j-1}+2,n_{j-1}+3}\right)^{q} \delta^{-(k-1)}
$$
for $1\leq k\leq p_j^{r_j}-1$. The element $e_{j,0,1}$ is fixed by the action for all $1\leq j\leq t$. 
Since $\left(A_{n_{j-1}+1,n_{j-1}+2}\right)^{q}$ does not belong to ${\cal B}$ then we rewrite it using elements of ${\cal B}$ 
\begin{align*}
\left(A_{n_{j-1}+1,n_{j-1}+2}\right)^{q} &=e_{j,0,1}^{q}
\prod_{1\leq k\leq p_j^{r_j}-1} \left(\delta^{k-1} \left(A_{n_{j-1}+2,n_{j-1}+3}\right)^{q} \delta^{-(k-1)}\right)^{-1}\\
&=e_{j,0,1}^{q}\cdot\displaystyle\prod_{1\leq k\leq p_j^{r_j}-1}e_{j,0,k+1}^{-1}
\end{align*}
Hence, in this case, the action may be described using an arrows diagram as 
$e_{j,0,2}\mapsto e_{j,0,3}\mapsto e_{j,0,4}\mapsto \cdots \mapsto e_{j,0,p_j^{r_j}-1}\mapsto e_{j,0,p_j^{r_j}}\mapsto e_{j,0,1}^{q}\cdot\displaystyle\prod_{1\leq k\leq p_j^{r_j}-1}e_{j,0,k+1}^{-1}$ and $e_{j,0,1}\mapsto e_{j,0,1}$, for all $1\leq j\leq t$.

\item For each $1\leq j\leq t$ there are $\frac{p_j^{r_j}-3}{2}$ orbits of size $p_j^{r_j}$ given by
$$
e_{j,h,k}=\delta^{k-1} A_{j,h+2}^{q} \delta^{-(k-1)} 
$$
for $1\leq k\leq p_j^{r_j}$  and $1\leq h\leq \frac{p_j^{r_j}-3}{2}$.

\item For $t+1\leq j<h \leq 2t$ let
$$
e_{j,h,k}=\delta^{k-1} A_{n_{j-t-1}+1,n_{h-t-1}+1}^{q} \delta^{-(k-1)} 
$$
for $1\leq k\leq p_j^{r_j}p_h^{r_h}$.
\end{itemize}
 
Consider the following (lexicographic) ordered basis of $L=\Z^{\frac{n(n-1)}{2}}$
\begin{align}\label{eq:basisB}
{\cal B}=&\set{e_{j,h,k}}{1\leq j\leq t,\, 0\leq h\leq \frac{p_j^{r_j}-3}{2},\, 1\leq k\leq p^r}\\
& \cup \set{e_{j,h,k}}{1\leq j<h \leq t,\, 1\leq k\leq p_j^{r_j}p_h^{r_h}}.\nonumber
\end{align}
Let  
$$
{\cal N}_{z}[q]=\left(
\begin{array}
[c]{ccccccc}%
1 & 0 & 0 & \cdots & 0 & 0 & q\\
0 & 0 & 0 & \cdots & 0 & 0 & -1\\
0 & 1 & 0 & \cdots & 0 & 0 & -1\\
0 & 0 & 1 & \cdots & 0 & 0 & -1\\
\vdots & \vdots & \vdots & \ddots & \vdots & \vdots & \vdots\\
0 & 0 & 0 & \cdots & 1 & 0 & -1\\
0 & 0 & 0 & \cdots & 0 & 1 & -1
\end{array}
\right)
\textrm{ and }
{\cal M}_{z}=\left(
\begin{array}
[c]{ccccccc}%
0 & 0 & 0 & \cdots & 0 & 0 & 1\\
1 & 0 & 0 & \cdots & 0 & 0 & 0\\
0 & 1 & 0 & \cdots & 0 & 0 & 0\\
0 & 0 & 1 & \cdots & 0 & 0 & 0\\
\vdots & \vdots & \vdots & \ddots & \vdots & \vdots & \vdots\\
0 & 0 & 0 & \cdots & 1 & 0 & 0\\
0 & 0 & 0 & \cdots & 0 & 1 & 0
\end{array}
\right)
$$
square matrices of order $z$. 
Therefore, using the ordered basis ${\cal B}$ of $L=\Z^{\frac{n(n-1)}{2}}$, we may conclude from the information above that we obtain the following matrix description of the holonomy representation of $\Gamma_{q}$ given in \req{repzq}.

\begin{thm}\label{th:mainrep}
Let $\Gamma_{q}$ be the Bieberbach group constructed in \reth{akthm} with holonomy group $\Z_q$ and holonomy representation	$\psi_{q}\colon \Z_{q}\to \operatorname{Aut}\left(\Z^{\frac{n(n-1)}{2}}\right)$.
	Let ${\cal B}$ be  the ordered basis of the free Abelian group $L=\Z^{\frac{n(n-1)}{2}}$ given in \req{basisB}. Then we obtain the following matrix decomposition of $\psi_{q}$:
\begin{equation}\label{eq:matrixzpr}
\lbrack\psi_{q}(\theta)]_{{\cal B}}=
\bigoplus_{1\leq j\leq t}\left({\cal N}_{p_j^{r_j}}[q]\oplus \left[\bigoplus_{1\leq h\leq \frac{p_j^{r_j}-3}{2}}{\cal M}_{p_j^{r_j}}\right] \right)
\oplus \left(\bigoplus_{1\leq j<h \leq t} {\cal M}_{p_j^{r_j}p_h^{r_h}}\right)
\end{equation}
In the case $q=3$ \req{matrixzpr} has the form $\lbrack\psi_{3}(\theta)]_{{\cal B}}={\cal N}_{3}[3]$.
The characteristic polynomial of $\lbrack\psi_{q}(\theta)]_{{\cal B}}$ is
$$
\prod_{1\leq j\leq t} \left(x^{p_j^{r_j}}-1\right)^{\frac{p_j^{r_j}-1}{2}}\cdot \prod_{1\leq j<h \leq t} \left(x^{p_j^{r_j}p_h^{r_h}}-1\right)
$$ 
\end{thm}

\begin{proof}
The matrix decomposition of $\psi_{q}$ given in \req{matrixzpr} follows from the description of the action by conjugacy of $\Z_q$ over ${\cal B}$ given in the first part of this subsection. Rest to compute the characteristic polynomial of $\lbrack\psi_{q}(\theta)]_{{\cal B}}$. 
Removing the first line and column of the matrix ${\cal N}_{z}[q]$ we obtain the companion matrix of the monic polynomial $1+x+x^2+\cdots +x^{z-1}$ , so the characteristic polynomial of the matrix ${\cal N}_{z}[q]$ is $(x-1)(1+x+x^2+\cdots +x^{z-1})=x^{z}-1$. The matrix ${\cal M}_{z}$ is the companion matrix of the monic polynomial $x^{z}-1$. The result follows from these observations.
\end{proof}

We shall use \reth{mainrep} to prove \reth{mainzpr}, but first we need some definitions.

\begin{defn}
A diffeomorphism $f\colon M \rightarrow M$ from a Riemannian manifold into itself
is called \textit{Anosov} and that $M$ has an hyperbolic structure if the
following condition is satisfied: there exists a splitting of the tangent
bundle $T(M)=E^{s}+E^{u}$ such that $Df\colon E^{s} \rightarrow E^{s}$ is
contracting and $Df\colon E^{u} \rightarrow E^{u}$ is expanding. 
\end{defn}

Classification of all compact manifolds supporting Anosov diffeomorphism
play an important r\^ole in the theory of dynamical systems, that notion
represents a kind of global hyperbolic behavior, and provides examples of
stable dynamical systems. 
The problem of classifying those compact manifolds (up to diffeomorphism) which admit an Anosov diffeomorphism was first proposed by S.~Smale \cite{Sm}. H.~Porteous gave an interesting algebraic characterisation of the existence of Anosov diffeomorphisms for flat manifolds that just depends on the holonomy representation, see \cite[Theorems 6.1 and 7.1]{P}.
Now, using the information above, we are able to prove \reth{mainzpr}.

\begin{proof}[Proof of \reth{mainzpr}]
Let $q=p_1^{r_1} p_2^{r_2}\cdots p_t^{r_t}$ be an odd number, where $p_i^{r_i}$ are distinct odd primes and $r_i\geq 1$ for all $1\leq i\leq t$. 
From \reth{akthm} let $\mathcal{X}_{\Gamma_{q}}$ be the  flat manifold of dimension $\frac{n(n-1)}{2}$ with fundamental group 
$\Gamma_{q}\leq B_{n}/[P_{n}, P_{n}]$ and holonomy group $\Z_{q}$. 

\begin{enumerate}

\item Since the determinant of the matrix (representation) given in \req{matrixzpr} is equal to 1 then follows from \rerem{orientable} that the flat manifold $\mathcal{X}_{\Gamma_{q}}$ is orientable.

\item From \cite[Theorems~6.4.12 and 6.4.13]{Dekimpe} we may compute the first Betti number of $\mathcal{X}_{\Gamma_{q}}$ as 
$\beta_1(\mathcal{X}_{\Gamma_{q}})=\frac{n(n-1)}{2}-\operatorname{rank}\left(  [\psi_{\Gamma_{q}}(\theta)]_{\cal B}-I \right)= \displaystyle\sum_{i=1}^t\frac{p_i^{r_i}-1}{2} + \frac{t(t-1)}{2}$
where $[\psi_{\Gamma_{q}}(\theta)]_{\cal B}$  is the matrix given in \req{matrixzpr} and $I$ is the identity matrix of order $\frac{n(n-1)}{2}$.

\item Since $q$ is odd then $-1$, $i$, $-i$ are not eigenvalues of the matrix $\lbrack\psi_{q}(\theta)]_{{\cal B}}$. 
In \reth{mainrep} we computed the characteristic polynomial of $\lbrack\psi_{q}(\theta)]_{{\cal B}}$ and from it we conclude that if the numbers $1$,  $\omega$, $\omega^{2}$, $-\omega$ and $-\omega^{2}$ (where $\omega^{3}=1$) are eigenvalues of the matrix $\lbrack\psi_{q}(\theta)]_{{\cal B}}$ then they have multiplicity greater than one, except in the case in which $q=3$. In the case $q=3$ the eigenvalues of the matrix $\lbrack\psi_{3}(\theta)]_{{\cal B}}$ are the cubical roots of the unity, with multiplicity one. Hence, from \cite[Theorem~7.1]{P} the flat manifold $\mathcal{X}_{\Gamma_{q}}$ admits Anosov diffeomorphism if and only if $q\neq 3$.
\end{enumerate}
\end{proof}

Now we use Johnson's method \cite{Johnson} to give a presentation for the group $\Gamma_{q}$ of \reth{mainzpr}.

\begin{prop}\label{prop:pres}
Let $\Z_{q}\subseteq \sn$. The Bieberbach group $\Gamma_{q}$ fits into a short exact sequence
\begin{equation*}
1 \to  \Z^{\frac{n(n-1)}{2}} \to  \Gamma_{q} \stackrel{\overline{\sigma}}{\to} \Z_{q} \to 1
\end{equation*}
and have a presentation given by:
\begin{enumerate}
\item Generators: 
\begin{enumerate}
\item $e_{j,h,k}\in {\cal B}$ (see \req{basisB}),
\item $A_{1,2}\delta_{0,n_1},\, A_{n_1+1,n_1+2}\delta_{n_1,p_2^{r_2}},\,\cdots,\, A_{n_{t-1}+1,n_{t-1}+2}\delta_{n_{t-1},p_t^{r_t}}$.
\end{enumerate}

\item Relations:
\begin{enumerate}
\item $[e_{j,h,k},\, e_{r,u,s}]=1$, for all $e_{j,h,k},\, e_{r,u,s}\in {\cal B}$,

\item $\left(A_{n_{j-1}+1,n_{j-1}+2}\delta_{n_{j-1},p_{j}^{r_{j}}}\right)^{p_{j}^{r_{j}}}=e_{j,0,1}$, for $1\leq j\leq t$.

\item 
\begin{align*}
&\left(A_{n_{j-1}+1,n_{j-1}+2}\delta_{n_{j-1},p_{j}^{r_{j}}}\right) e_{j,h,k}\left(A_{n_{j-1}+1,n_{j-1}+2}\delta_{n_{j-1},p_{j}^{r_{j}}}\right)^{-1}\\
&=
\begin{cases}
e_{j,h,k}, & \textrm{if } 1\leq j\leq t,\, h=0,\, k=1\\
e_{j,0,1}^{q}\cdot\displaystyle\prod_{1\leq k\leq p_j^{r_j}-1}e_{j,0,k+1}^{-1}, & \textrm{if } 1\leq j\leq t,\, h=0,\, k=p_j^{r_j}\\
e_{j,0,k+1}, & \textrm{if } 1\leq j\leq t,\, h=0,\, 2\leq k\leq p_j^{r_j}-1\\
e_{j,h,k+1}, & \textrm{if } 1\leq j\leq t,\, 1\leq h\leq \frac{p_j^{r_j}-3}{2},\, 1\leq k\leq p_j^{r_j}\\
e_{j,h,k+1}, & \textrm{if }  t+1\leq j<h \leq 2t,\, 1\leq k\leq p_j^{r_j}p_h^{r_h}\end{cases}
\end{align*}
\end{enumerate}
\end{enumerate}
taking $k+1 (mod p_j^{r_j})$ or $k+1 (mod p_j^{r_j}p_h^{r_h})$ in the respective case.
\end{prop}

\begin{proof}
The short exact sequence follows from the construction given in \reth{akthm}. The presentation was obtained applying the method of \cite[Chapter~10]{Johnson}. Relations $(ii)$ follows using the action by conjugacy:
\begin{align*}
\left(A_{n_{j-1}+1,n_{j-1}+2}\delta_{n_{j-1},p_{j}^{r_{j}}}\right)^{p_{j}^{r_{j}}} = \prod_{k=1}^{p_j^{r_j}} \delta_{n_{j-1},p_{j}^{r_{j}}}^{k-1} \left(A_{n_{j-1}+1,n_{j-1}+2}\right) \delta_{n_{j-1},p_{j}^{r_{j}}}^{-(k-1)}
=e_{j,0,1}
\end{align*}
The conjugacy part (relations $(iii)$) was described in the computations given before \req{basisB}.
\end{proof}

Next, we compute the rank and give a basis of the center of $\Gamma_{q}$.

\begin{thm}\label{th:centerzdk}
Let $\Gamma_{q}$ be the Bieberbach subgroup of $B_{n}/[P_{n},P_{n}]$ of dimension ${n}({n}-1)/2$, with holonomy group $\Z_{q}$. 
Then the center of the Bieberbach group $\Gamma_{q}$, $\mathcal{Z}(\Gamma_{q})$, is the free abelian group of rank $\displaystyle\sum_{i=1}^t\frac{p_i^{r_i}-1}{2} + \frac{t(t-1)}{2}$:
\begin{align*}
\mathcal{Z}(\Gamma_{q})&= 
\left(
\bigoplus_{1\leq j\leq t}\left( \Z[e_{j,0,1}] \right)
\right) \oplus
\left(
\bigoplus_{ \substack{1\leq j\leq t\\
1\leq h\leq \frac{p_j^{r_j}-1}{2}} }\left(\mathbb{Z}\left[ \prod_{1\leq k\leq p_j^{r_j}} e_{j,h,k} \right]\right)
\right)\\
&\oplus
\left(
\bigoplus_{ t+1\leq j<h \leq 2t }\left(\mathbb{Z}\left[ \prod_{1\leq k\leq p_j^{r_j}p_h^{r_h}} e_{j,h,k} \right]\right)
\right)
\end{align*}
where the elements $e_{*,*,*}$ belongs to the basis ${\cal B}$.
\end{thm}

\begin{proof}
The expression of $\mathcal{Z}(\Gamma_{q})$ follows from  Lemma~5.2(3) of \cite{Sz} that asserts that the center of $\Gamma_{q}$ is $(\Z^{\frac{n(n-1)}{2}})^{\Z_{q}}$. That means that $\mathcal{Z}(\Gamma_{q})$ is the group given by the fixed elements of the action by conjugacy of $\Z_{q}$ over the group $\Z^{\frac{n(n-1)}{2}}$. Hence the rank of $\mathcal{Z}(\Gamma_{q})$ corresponds to the number of orbits of the action by conjugacy of $\Z_q$ on the basis ${\cal B}$ given in \req{basisB} and a basis of $\mathcal{Z}(\Gamma_{q})$ is obtained taking the product of all the elements in each orbit. 
The result follows analising the action described in the computations given before \req{basisB} (that is equivalent to analise the fixed elements of the matrix $\lbrack\psi_{q}(\theta)]_{{\cal B}}$ given in \req{matrixzpr}).
\end{proof}

\begin{rem}
Notice that we may quotient $\Gamma_q$ by subgroups of the center of $\Gamma_q$ that are invariant under the action by conjugacy obtaining crystallographic groups (not necessarily torsion free). 
In particular, if we consider the quotient of $\Gamma_q$ by its center we obtain the group of inner automorphisms of $\Gamma_q$ (see \cite[Theorem 5.1, page 65]{Sz}) $\Gamma_q/\mathcal{Z}(\Gamma_{q})=\operatorname{Inn}(\Gamma_q)$ that is a crystallographic group of dimension 
$$
\frac{n(n-1)}{2} - \operatorname{rank}(\mathcal{Z}(\Gamma_{q}))= \sum_{i=1}^t\frac{p_i^{r_i}(p_i^{r_i}-1)}{2} + \sum_{1\leq \ell<s\leq t}p_{\ell}^{r_{\ell}}p_s^{r_s} - \sum_{i=1}^t\frac{p_i^{r_i}-1}{2} - \frac{t(t-1)}{2}
$$
that fits into a short exact sequence 
$1 \to  \frac{\Z^{\frac{n(n-1)}{2}}}{\mathcal{Z}(\Gamma_{q})} \to  \frac{\Gamma_{q}}{\mathcal{Z}(\Gamma_{q})} \stackrel{\overline{\sigma}}{\to} \Z_{q} \to 1$.
\end{rem}

\subsection{$\Z_{p^r}$-manifolds with $p$ an odd prime number}\label{subsec:zpr}

Let $n=p^r$ where $p$ is an odd prime and $r\geq 1$. That means that $n$ is the odd number $q$ of the first paragraph of this section taking $t=1$. 
Let $\Z_{p^r}\subseteq \sn$ be the cyclic group generated by the permutation $\theta=(p^r, p^r-1,\ldots,3,2,1)\in \sn$.  
Let $\Gamma_{\Z_{p^r}}\leq B_{p^r}/[P_{p^r}, P_{p^r}]$ be the Bieberbach group constructed in \reth{akthm}. 
For short we just write $\Gamma_{p^r}$ to indicate the group $\Gamma_{\Z_{p^r}}$.
The Bieberbach group $\Gamma_{p^r}$ fits into a short exact sequence
$
1\to \Z^{\frac{p^r(p^r-1)}{2}} \to \Gamma_{p^r} \to \Z_{p^r} \to 1
$
and have holonomy representation
\begin{equation}\label{eq:repzpr}
\psi_{p^r}\colon \Z_{p^r}\to \operatorname{Aut}\left(\Z^{\frac{p^r(p^r-1)}{2}}\right)
\end{equation}
induced from the action by conjugacy of $\Z_{p^r}$ over the group $\Z^{\frac{p^r(p^r-1)}{2}}$.

We recall that a  K\"ahler manifold is a $2n$-real manifold endowed with Riemannian metric, complex structure, and a sympletic structure which are compatible at every point.	
	A finitely presented group is a  K\"ahler group if it is the fundamental group of a closed  K\"ahler manifold.
 For more details about K\"ahler manifolds see \cite{DeHaSz}, \cite{JR} and also \cite[Section~7]{Sz}. 
In the proof of \reth{kahler} we will apply an algebraic characterisation of K\"ahler manifolds (see \cite[Propositions~7.1~and~7.2]{Sz}) in order to decide which flat manifolds of even dimension arising from Artin braid groups admit K\"ahler geometry.

\begin{proof}[Proof of \reth{kahler}]
Let $p$ be an odd prime and let $r\geq 1$. 
Let $\mathcal{X}_{\Gamma_{p^r}}$ be the  flat manifold of dimension $\frac{p^r(p^r-1)}{2}$ with fundamental group 
$\Gamma_{p^r}\leq B_{p^r}/[P_{p^r}, P_{p^r}]$ and holonomy group $\Z_{p^r}$. 
We need to determine in which conditions the coefficients of the real irreps that appears in the representation of $\psi_{\Gamma_{p^r}}$ are even, and then apply \cite[Proposition~7.2]{Sz} to verify if 
the representation is essentially complex.
From \req{matrixzpr} we conclude that the character vector of $\psi_{p^r}$ is
\begin{equation*}
\overset{\rightarrow}{\chi}_{\psi_{p^r}}=\left(
\begin{array}
[c]{c}%
\frac{p^r\left(p^r-1\right)}{2} \\
0\\
\vdots\\
0\\
\vdots\\
0\\
\end{array}
\right). 
\end{equation*}
Then the holonomy representation $\psi_{p^r}$ has the property that each real irrep appears $\frac{p^r-1}{2}$ times in its decomposition.
Hence, the coefficients of the real irreps that appears in the representation of $\psi_{\Gamma_{p^r}}$ are even if and only if there is an integer $u\geq 1$ such that $p^r=4u+1$. 
Therefore, follows from \cite[Propositions~7.1~and~7.2]{Sz} that the representation given in \req{repzpr} is essentially complex and the flat manifold $\mathcal{X}_{\Gamma_{p^r}}$ is K\"ahler if and only if there is an integer $u\geq 1$ such that $p^r=4u+1$.
\end{proof}

In the last 20 years, after the solution of Calabi's conjecture by S.~T.~Yau, studies related to complex manifolds have grown. One of the most important classes are the so-called \textit{Calabi-Yau manifolds}, that is K\"ahler manifolds with holonomy group contained in $SU(n)$, see \cite[Section~7.2]{Sz} and \cite{DeHaSz}. 
Follows from \reth{mainzpr} that the flat manifold $\mathcal{X}_{\Gamma_{p^r}}$ of \reth{kahler} is orientable, hence we obtain families of Calabi-Yau manifolds with fundamental group a Bieberbach subgroup of the quotient of the Artin braid group $B_{p^r}/[P_{p^r},P_{p^r}]$.
\begin{cor}\label{cor:kahler}
Let $p$ be an odd prime such that $p^r=4u+1$ for some $u,r\geq1$. Then the flat manifold  $\mathcal{X}_{\Gamma_{p^r}}$ of \reth{kahler} is a Calabi-Yau flat manifold of dimension $2u(4u+1)$.
\end{cor}

\end{document}